\documentclass[10pt,oneside,leqno]{amsart}

\usepackage{amsfonts,latexsym,rawfonts,amsmath,amssymb,amsthm, mathrsfs, lscape}
\usepackage[all]{xy}
\usepackage[english]{babel}
\usepackage[utf8]{inputenc}

\usepackage{array, tabularx}

\usepackage{setspace}
\usepackage{textcomp}

\usepackage[hypertexnames=false,
backref=page,
    pdftex,
    pdfpagemode=UseNone,
    breaklinks=true,
    extension=pdf,
    colorlinks=true,
    linkcolor=blue,
    citecolor=blue,
    urlcolor=blue,
]{hyperref}

\usepackage[top=1.5in, bottom=.9in, left=0.9in, right=0.9in]{geometry}

 \newtheorem{thm}{Theorem}[section]
 \newtheorem{prop}[thm]{Proposition}
 \newtheorem{cor}[thm]{Corollary}
 \newtheorem{lemma}[thm]{Lemma}
 
\theoremstyle{definition}
 \newtheorem{defi}[thm]{Definition}
 \newtheorem{rem}[thm]{Remark}

\newcommand{\R}{\mathbb{R}}

\newcommand{\kth}[1]{\ifthenelse{\equal{#1}{1}}{$#1^\text{st}$}{\ifthenelse{\equal{#1}{2}}{$#1^\text{nd}$}{\ifthenelse{\equal{#1}{3}}{$#1^\text{rd}$}{$#1^\text{th}$}}}}

\newcommand{\Leb}[2]{\ifthenelse{\equal{#2}{0}}{\mathrm{L}^{2}\left(X;\wedge^{#1}T^*X\right)}{\mathrm{L}^{2}_{#2}\left(X;\wedge^{#1}T^*X\right)}}
\newcommand{\LebK}[2]{\ifthenelse{\equal{#2}{0}}{\mathrm{L}^{2}\left(K;\wedge^{#1}T^*X\right)}{\mathrm{L}^{2}_{#2}\left(K;\wedge^{#1}T^*X\right)}}

\newcommand{\Cinf}[1]{\ifthenelse{\equal{#1}{0}}{\mathcal{C}^\infty\left(X;\R\right)}{\mathcal{C}^\infty\left(X;\wedge^{#1}T^*X\right)}}
\newcommand{\Cinfc}[1]{\ifthenelse{\equal{#1}{0}}{\mathcal{C}^\infty_{\mathrm{c}}\left(X;\R\right)}{\mathcal{C}^\infty_{\mathrm{c}}\left(X;\wedge^{#1}T^*X\right)}}
\newcommand{\Sob}[3]{\ifthenelse{\equal{#2}{0}}{\mathrm{W}^{#3, 2}\left(X;\wedge^{#1}T^*X\right)}{\mathrm{W}^{#3, 2}_{#2}\left(X;\wedge^{#1}T^*X\right)}}

\newcommand{\SobK}[3]{\ifthenelse{\equal{#2}{0}}{\mathrm{W}^{#3, 2}\left(K;\wedge^{#1}T^*X\right)}{\mathrm{W}^{#3, 2}_{#2}\left(K;\wedge^{#1}T^*X\right)}}

\allowdisplaybreaks[1]

\title[A characterization of projectively flat manifolds]{Positive projectively flat manifolds are locally conformally flat-K\"ahler Hopf manifolds}
\author{Simone Calamai}
\address{(S.~Calamai) Dip. di Matematica e Informatica ``U. Dini'' - Universit\`a di Firenze \endgraf Viale Morgagni 67A -  Firenze - Italy}
\email{simocala at gmail.com}

 \keywords{Projectively flat, locally conformally flat-K\"ahler, Boothby metric}
 \subjclass[2010]{Primary: 53C07. Secondary: 53C55.}

\begin{document}

\begin{abstract}
We define a partition of the space of projectively flat metrics in three classes according to the sign of the Chern scalar curvature; we prove that 
the class of negative projectively flat metrics is empty, and that the class of positive projectively flat metrics consists precisely of 
locally conformally flat-K\"ahler metrics on Hopf manifolds, explicitly characterized by Vaisman \cite{vaisman}. 
Finally, we review the known characterization and properties of zero projectively flat metrics.
As applications, we make sharp a list of possible projectively flat metrics by Li, Yau, and Zheng \cite[Theorem 1]{liyauzhengcag94}; moreover
we prove that projectively flat astheno-K\"ahler metrics are in fact K\"ahler and globally conformally flat.
\end{abstract}

\maketitle

\section*{Introduction}
A projectively flat metric $\omega$ on a given compact complex manifold $M$ of complex dimension $n$ is in particular a 
 Hermitian Yang Mills metric.
The latter are solutions of the equation $\Lambda_g F_h = \gamma \cdot Id_E$, where $(E,h)\rightarrow (M,g)$ 
is a complex rank $r$ Hermitian holomorphic vector bundle, $\gamma$ is a real valued function on $M$, $F_h$ is the Chern-curvature of $h$, and $\Lambda_\omega F_h$ is the mean curvature.
In that environment, projectively flat metrics are solutions of the equation $F_h = \frac{1}{r} tr_{h}F_h \cdot Id_E$. 
 Our present concern consists of the case when $E$ is the holomorphic tangent bundle $T_M$, so hereafter by projectively flat metric we mean
 a solution of the equation 
 \[
 F_h= \frac{1}{n} tr_h F_h \cdot Id_{T_M} \; .
 \]
Hermitian Yang Mills metrics are in bijection to stable vector bundles via the Kobayashi-Hitchin correspondence thanks to the works of 
Donaldson for algebraic surfaces \cite{donaldsonalgebraicsurfaces} and manifolds \cite{donaldsonalgebraicmanifolds}, Uhlenbeck Yau for K\"ahler
manifolds \cite{uhlenbeckyau}, Buchdahl \cite{buchdahl} for surfaces, and Li Yau \cite{liyau} for Hermitian manifolds.
 In particular, uniqueness theorems of that theory yield that there is at most one projectively flat metric on a given compact complex manifold; 
 we should emphasize here that, since a globally conformal metric of a projectively flat metric is again projectively flat, 
 we understand uniqueness modulo global conformal transformations of the metric.
In complex dimension one, every compact Riemann surface $S$ is projectively flat, and since all Hermitian metrics on 
$S$ are conformal to each other, we can think of any such metric to be projectively flat.

 In complex dimension two there is a complete understanding of projectively flat metrics as well, 
 which is presented for instance in \cite[page 180]{lubketeleman}. In fact, projectively flat complex surfaces are precisely those complex surfaces 
 admitting a locally conformally flat-K\"ahler metric: on their turn, the latter are either K\"ahler flat surfaces 
 (which are classified in \cite{dekimpe}) 
 or locally conformally flat-K\"ahler Hopf surfaces (which are explicitly characterized by Vaisman \cite{vaisman}).
 
 For higher complex dimension the previous picture fails; 
 in fact generalized Iwasawa manifolds admit Chern-flat (and hence, projectively flat) 
 metric \cite{balas,corderofernandezgray,matsuo96} which is not locally conformally flat-K\"ahler.
 (This can be seen as follows: Vaisman \cite{vaisman}, as described in \cite[page 180]{lubketeleman},
 proved that the only locally conformally K\"ahler flat metrics are either Hopf or K\"ahler flat, and clearly Iwasawa aren't such.) This leads  
 to a partition  of the family of projectively flat metrics in classes that have both 
 a geometrical meaning and for some of which we can get a complete description.
 
 Whence, we consider the conformal class of a projectively flat metric $\{h\}$ and we look at the sign of its Gauduchon degree. 
 By the Gauduchon conformal method \cite{angellacalamaispotti,balas,gauduchon,yang}, 
 there exists in $\{h\}$ a representative whose scalar curvature with respect to the Chern connection
 has definite sign which is the same as the sign of the Gauduchon degree. 
  This leads to the natural definition  of positive (respectively zero, or negative) projectively flat metric
  (Definition \ref{defi:negzeroposprojflat}).
 
 We are going to prove, in Theorem \ref{thm:positiveprojectivelyflat}, that the class of positive projectively flat metrics consists precisely of 
 locally conformally flat-K\"ahler metrics on Hops manifolds 
 (which are classified with explicit description by Vaisman \cite{vaisman}).
 A key step for getting the result is Lemma \ref{lemma:pfimpliesbalancedorlcK}, for which we give a proof alternative to
 the original argument of Li, Yau, and Zheng \cite[Lemma 2]{liyauzhengcag94}.
 
 Concerning negative projectively flat metrics, we prove that they do not exist (Theorem \ref{thm:negativeprojectivelyflat}, compare \cite[Theorem 4.4]{matsuo96}).
 
 About the class of zero projectively flat metrics, we give a characterization (Theorem \ref{thm:zeroprojectivelyflat}) of them
 which involves balanced metrics and globally conformally Chern flat metrics, essentially by means of the existing literature 
 (\cite{liyauzhengcag94,matsuo96}). Then we recall that the understanding of Chern flat metrics 
 is fairly satisfactory thanks to a result by Boothby \cite{boothby}.
  The class of zero projectively flat metrics can be on its turn subdivided into K\"ahlerian and non-K\"ahlerian manifolds.
 About the first ones, we give a characterization criterion (Theorem \ref{thm:zeroprojectivelyflat2}) 
 building on the Calabi Yau theorem and the Kobayashi L\"ubke inequality; we recall that they are classified up to dimension 
 three (see \cite{dekimpe}).

  As applications of this partition of projectively flat metrics we prove (Corollary \ref{cor:refinement}) 
  a refinement of \cite[Theorem 1]{liyauzhengcag94} by Li, Yau, and Zheng of possible
  occurrences of projectively flat manifolds. In fact, among all the finite undercovers of Hopf manifolds, 
  we are able to say that only the locally conformally flat-K\"ahler Hopf manifolds classified 
  by Vaisman admit projectively flat metric, which is the Boothby metric.
  
  A second application (Corollary \ref{cor:asthenoarekahler}) is that if a metric is 
  astheno-K\"ahler and projectively flat, 
  then it is K\"ahler and zero projectively flat.
  
  \subsection*{Acknowledgements}
  The author is supported by SIR 2014 AnHyC ``Analytic aspects in complex and hypercomplex geometry" (code RBSI14DYEB) and by GNSAGA of INdAM; he also wants to thank Xiuxiong Chen for constant support. 
  Thanks to Song Sun for his support and for recommending the reading of \cite{lubketeleman}, and to Alexandra Otiman and David Petrecca for pointing out reference \cite{dragomirornea}.
  This research received benefit by the great environment at Stony Brook and the Simons Center for Geometry and Physics. During this  research the author's visiting assistant professor at the Institute of Mathematical Sciences of the ShanghaiTech University, and funded by the Municipal Science and Technology Commission.
  
\section{Setup and definitions}
 Let $(M, J)$ be a differentiable manifold of complex dimension $n$ endowed with an integrable complex structure; 
 also, $T_M$ will denote the holomorphic tangent bundle of $M$. 
 The integrability of $J$ is required since it is usually assumed in the theory of Hermitian Yang Mills metrics (see \cite{lubketeleman}).  
 Given a Hermitian metric $h$ on $M$, on a coordinate chart of $M$ the curvature tensor induced by the Chern connection is given by
 \begin{align}\label{equa:curvaturetensor}
  \Theta_{i\bar{j}k\bar{l}} := - \frac{\partial^2 h_{i\bar{j}}}{\partial z^k \partial \bar{z}^l} 
  + h^{p\bar{q}} \frac{\partial h_{i\bar{q}}}{\partial z^k} \frac{\partial h_{p\bar{j}}}{\partial \bar{z}^l} ,
 \end{align}
  which leads to three types of Ricci tensors that we will refer to as first (respectively second, third) Ricci, 
  as follows
  \begin{align} \label{equa:riccitensors}
   \Theta^{(1)}_{k\bar{l}} := h^{i\bar{j}} \Theta_{i\bar{j}k\bar{l}} ; \; \; \; \; \; \; \; \; \; \; 
   \Theta^{(2)}_{i\bar{j}} := h^{k\bar{l}} \Theta_{i\bar{j}k\bar{l}} ; \; \; \; \; \; \; \; \; \; \; 
   \Theta^{(3)}_{i\bar{l}} := h^{k\bar{j}} \Theta_{i\bar{j}k\bar{l}} \; .
  \end{align}
 Hoping that the following notation is suggestive rather than confusing, we will also denote the curvature tensor of $h$ as 
$\Theta : = :\Theta_{[h]}$, the first Ricci tensor as $\Theta^{(1)}_{[h]} :=: tr_h \Theta_{[h]}$, 
and the second Ricci tensor as $\Theta^{(2)}_{[h]} :=: \Lambda_h \Theta_{[h]}$

Here we introduce the main focus of the present manuscript
 
 \begin{defi}\label{defi:projectively flat metric}
  A projectively flat metric $h$ is a Hermitian metric on $M$ such that it satisfies
 \begin{align}\label{equa:projflatmetric}
  \Theta_{i\bar{j}k\bar{l}} = \frac{1}{n} \Theta^{(1)}_{k\bar{l}} h_{i\bar{j}} \; .
 \end{align}
 Also, \eqref{equa:projflatmetric} can be expressed as $\Theta_{[h]} = \frac{1}{n} tr_h \Theta_{[h]} \cdot h$.
 \end{defi}
 
 \begin{rem}
   In \cite[(2.2.3), page 51]{lubketeleman} are labeled as projectively flat metrics those $h$ of a Hermitian holomorphic vector bundle
   $(E,h)\rightarrow (M, g)$ for which there holds an equation that, 
   in the case $E=T_M$ and $h=g$ is precisely our \eqref{equa:projflatmetric}. The same metrics as in Definition \ref{defi:projectively flat metric} 
   were labeled as projectively flat already in \cite{liyauzhengcag94}.
 \end{rem}
 
 \begin{rem}
  It is straightforward to check that if $h$ is projectively flat, then given any 
  $u \in C^{\infty}(M; \mathbb{{R}})$, also $\exp(u)h$ is 
  projectively flat.
 \end{rem}

 \section{A partition for the class of projectively flat metrics}
 
 Our next goal is to define a partition on the class of projectively flat metrics. 
 We first need as lemma the Gauduchon conformal method 
 (which we state without proof, that can be found in \cite{angellacalamaispotti,balas,gauduchon,yang}), 
 for which it is useful to recall the following notion of Chern scalar curvatures
 
 \begin{defi}
  A further contraction of the Ricci tensors \eqref{equa:riccitensors} leads to two distinct types of Chern scalar curvatures, as follows
  \begin{align} \label{equa:scalarcurvatures}
   s:= s_{[h]} :=: \Lambda_h tr_h \Theta_{[h]} :=: h^{k\bar{l}}\Theta^{(1)}_{k\bar{l}} :=: h^{i\bar{j}} \Theta^{(2)}_{i\bar{j}} ; \; \; \; \; \; \; \; \; \; \; 
   \hat s : = {\hat s}_{[h]} :=: h^{i \bar{l}} \Theta^{(3)}_{i\bar{l}} \; .
  \end{align}
 \end{defi}

 Moreover, we recall the well known concept of the sign of the Gauduchon degree.
 
 \begin{defi}
  For any fixed conformal class $\{h\}$ of Hermitian metrics on a given compact complex manifold $M$
  of complex dimension $n\geq 2$.
  Recall from \cite{gauduchon} that there is one, up to homothety, Gauduchon metric $g\in \{ h\}$.
  Then the sign of the Gauduchon 
  degree is given by
  \begin{align*}
   \Gamma_{h}(M):= sign\left(\int_M s_g dV_g \right)\; .
  \end{align*}
 \end{defi}

 \begin{lemma} \label{lem: gauduchon conformal trick}
  Given a compact complex manifold $(M,J)$ and a conformal class $\{h\}$ of Hermitian metrics on $M$, there exists in $\{h\}$ a 
  representative $\tilde h$ such that its Chern scalar curvature $s_{\tilde h}$ has constant sign, which is the same as the sign of 
  the Gauduchon degree of $\{h\}$.
 \end{lemma}
 
 Next we introduce a fundamental definition for our purposes.
 
 \begin{defi} \label{defi:negzeroposprojflat}
  Let $(M, J)$ be a complex manifold, and assume the existence of the conformal class $\{h\}$ of projectively flat metrics on $M$. 
  Then, any representative of that class is called, respectively
  \begin{itemize}
   \item Negative projectively flat if and only if in $\{h\}$ there is a representative $\tilde h$ such that $s_{\tilde h}$ is a negative function.
   \item Zero projectively flat if and only if in $\{h\}$ there is a representative $\tilde h$ such that $s_{\tilde h}=0$.
   \item Positive projectively flat if and only if in $\{h\}$ there is a representative $\tilde h$ such that $s_{\tilde h}$ is a positive function.
  \end{itemize}
 \end{defi}

\begin{rem}
 Thanks to the fact that the sign of the Gauduchon degree is a conformal invariant, exactly one of the three possibilities in 
 Definition \ref{defi:negzeroposprojflat} occurs.
\end{rem}

\begin{rem} \label{rem:complexdimensionone}
 In complex dimension one, it is straightforward to check that any Hermitian metric on a Riemann surface is projectively flat.
 The partition described in Definition \ref{defi:negzeroposprojflat} amounts to the partition of Riemann surfaces according 
 to genus bigger or equal to two, genus one, and genus zero respectively.
\end{rem}

\section{Negative projectively flat metrics}
 The main result of this section is the non existence of negative projectively flat metrics; we remark that this result was essentially proved in
 \cite[Theorem 4.4]{matsuo96}, but it is worth noticing that according to the partition introduced in Definition \ref{defi:negzeroposprojflat}, 
 from this we will be able to conclude new applications. Moreover in \cite{matsuo96} Matsuo introduces the Chern counterpart of the Weyl
 and he proves a characterization of projectively flat metrics that involves the vanishing of that tensor; 
 on the other hand, the proof of non existence of negative projectively flat metrics doesn't need such characterization.
 
\begin{thm}\label{thm:negativeprojectivelyflat}
 Let $(M,J)$ be a compact complex manifold of complex dimension $n\geq 2$. Then it cannot admit a negative projectively flat metric.
\end{thm}
\begin{proof}
 Assume by contradiction that $h$ is a negative projectively flat metric on $M$. Applying $h^{k\bar j}$ 
 to \eqref{equa:projflatmetric} and summing over $k,j$ we end up with
 \begin{align}\label{equa:firstandthirdricci}
  n \cdot \Theta_{[h]}^{(3)} = \Theta_{[h]}^{(1)} \; .
 \end{align}
 As usual we can associate to $\Theta^{(a)}, a=1,2,3$, and $h$ their corresponding $(1,1)$-forms $\rho^{(a)}, a=1,2,3$, $\omega$.
 Hence, an equivalent statement of \eqref{equa:firstandthirdricci} is $n\cdot \rho^{(3)}=\rho^{(1)}$.
 The first and third Ricci forms satisfy the following general relation (see \cite{gauduchon,liuyang})
 \begin{align} \label{equa:thirdandfirstricci}
  \rho^{(3)} = \rho^{(1)} - \partial \partial^* \omega \; .
 \end{align}
We infer that, for any projectively flat metric there holds 
 \begin{align}\label{equa:ricciforanypj}
  \frac{n-1}{n}\rho^{(1)} = \partial \partial^* \omega \; .
 \end{align}
 Now, tracing by means of $\omega$ and integrating by parts we get
 \begin{align*}
  \frac{n-1}{n}\int_M s_{[h]} dV_h = \int_M  (\partial^* \omega, \, \partial^* \omega)_\omega dV_h \geq 0 \, ,
 \end{align*}
 which contradicts the assumption on the negativity of the Gauduchon degree of $h$. This completes the proof of the theorem.
\end{proof}

\begin{rem}
 As observed in Remark \ref{rem:complexdimensionone} there are examples of negative projectively flat metrics, which now turn out to
 be the only ones.
\end{rem}

\section{Zero projectively flat metrics}
 We begin this section with recalling a result in \cite{matsuo96}, which was also hinted in \cite{lubketeleman}; it says that
 projectively flat metrics are the same as locally conformally Chern flat metrics, in the sense now specify.

 \begin{defi}\label{defi:lcf}
  Let $M$ be a complex manifold, and $h$ be a Hermitian metric on $M$. Then $h$ is called
  locally conformally Chern flat if and only if for any point $p\in M$ there exists an open neighborhood $p\in U \subset M$
  and a function $u \in C^{\infty} (U ; \mathbb{R})$ such that 
  \begin{align} 
  \label{equa:lcf}
   \Theta_{[\exp(u) \cdot h]} = 0 \; .
  \end{align}
 \end{defi}
 
  \begin{lemma}
   \label{lemma:pjequallcf}
   Let $(M,\, h)$ be a compact Hermitian manifold. Then $h$ is projectively flat metric if and only if it is locally conformally Chern flat.
  \end{lemma}
  
  \begin{proof}
   We first assume that $h$ is projectively flat metric.
   Recall that the first Ricci form $\rho^{(1)}$ is a $d$-closed real $(1,1)$-form; then by the Dolbeault lemma, we can find an open
   neighborhood of any fixed $p\in M$ such that $\Theta^{(1)}_{k\bar l} = u_{k \bar l}$ for some $u \in C^{\infty}(U ; \mathbb{R})$.
   Then, over $U$ there holds, from the very definition of the curvature tensor $\Theta$,
\begin{align*}
 \Theta_{[\exp(u)\cdot h]} = \exp(u)\cdot  \left( \Theta_{i \bar j  k \bar l} - u_{k \bar l} \, \cdot \, h_{i \bar j} \right) \, ,
\end{align*}
 which entails, using \eqref{equa:projflatmetric} and that $\Theta^{(1)}_{k\bar l} = u_{k \bar l}$, 
  $\Theta_{[\exp(u)\cdot h]}=0$, that is, $h$ is locally conformally Chern flat.
  
  Now we assume that $h$ is locally conformally Chern flat; we are going to prove that on any point $p\in M$ there
  holds \eqref{equa:projflatmetric}.
  By hypothesis we have that around $p$, for some function $u$, there holds \eqref{equa:lcf}; expanding it we have
  \begin{align*}
   0=\exp(u) \cdot \left( \Theta_{i \bar j k \bar l}  - u_{k \bar l} \, \cdot \, h_{i \bar j} \right) \, , 
  \end{align*}
 which entails $ \Theta_{i \bar j k \bar l}  = u_{k \bar l} \, \cdot \, h_{i \bar j}$.
 Now contracting via $h$ along $i , \, j$ this is saying that $\Theta^{(1)}_{k \bar l} = n \, \cdot \, u_{k \bar l}$, and whence
 we can conclude that for $h$ there holds \eqref{equa:projflatmetric}.
 This completes the proof of the lemma.
   \end{proof}
   
 Our next goal is to characterize zero projectively flat metrics as global conformally Chern flat metrics, and also as those projectively flat metric
 which are balanced. 
 
 \begin{defi}
  Let $(M, \, h)$ be a  Hermitian manifold. Then $h$ is called globally conformally Chern flat if and only if it satisfies,
   for some function $v \in C^{\infty}(M ; \mathbb{R})$,
   \begin{align*}
    \Theta_{[\exp(v) \cdot h]} = 0 \; .
   \end{align*}
  Let $\omega$ be the $(1,1)$-form corresponding to $h$; then $h$ is called balanced if and only if $d(\omega^{n-1}) = 0$.
 \end{defi}
 
  \begin{prop} \label{thm:zeroprojectivelyflat}
   Let $(M, \, h)$ be a compact Hermitian manifold of complex dimension bigger or equal to two. The following facts are equivalent:
   \begin{enumerate}
    \item $h$ is zero projectively flat metric;
    \item $h$ is both projectively flat and balanced;
    \item $h$ is globally conformally Chern flat.
   \end{enumerate}
  \end{prop}

  \begin{proof}
We claim that (i) implies (ii). In fact, tracing the equation \eqref{equa:ricciforanypj} which holds for any projectively
flat metric, after integrating by parts we end up with
   \begin{align*}
    0= \frac{n-1}{n} \int_M s_{h} dV_h = \int_M  (\partial^* \omega, \, \partial^* \omega)_\omega dV_h \, ,
   \end{align*}
which implies that $\partial (\omega^{n-1}) = 0$, and whence $h$ is balanced.
   
   We claim that (ii) implies (iii). Since $h$ is balanced, from the general relation \eqref{equa:thirdandfirstricci} we conclude
   \begin{align*}
    \Theta^{(3)} = \Theta^{(1)} \; .
   \end{align*}
   On the other hand, since $h$ is projectively flat we infer again \eqref{equa:firstandthirdricci}. Whence, as $n\geq 2$,
   we conclude $\Theta^{(1)}= \Theta^{(3)}=0$, which plugged in \eqref{equa:projflatmetric} says that $h$ is Chern flat.

   The implication from (iii) to (i) being obvious, this completes the proof of the theorem.
  \end{proof}

\begin{rem}
 We emphasize that the equivalence between (ii) and (iii) in Theorem \ref{thm:zeroprojectivelyflat} is very well known and already
 present in literature.
\end{rem}

We now recall a classical result by Boothby \cite{boothby}, which makes the understanding of compact Chern flat manifolds fairly satisfactory.
 
 \begin{prop}
  \label{prop:boothby}
  Let $(M,\, h)$ be a compact Hermitian manifold. 
  Then $(M,\, h)$ is Chern flat if and only if $M$ is a compact quotient of a complex Lie group and $h$ is a left invariant metric.
 \end{prop}

 \begin{rem}
  An important subclass of compact Chern flat manifolds are the complex parallelizable manifolds, which were classified by Wang \cite{hcwang}.
 \end{rem}
  
On their turn, zero projectively flat manifolds can be subdivided between K\"ahlerian and not K\"ahlerian. Our next goal is to 
give an easy characterization of the K\"ahlerian zero projectively flat manifolds, building on the Calabi Yau theorem and on the
Kobayashi L\"ubke inequality.
 
As preparation, we need to state a result in \cite{matsuo96} whose proof can be found in there. 
 
\begin{lemma}\label{lem:preparation}
 Let $(M, \, h)$ be a projectively flat manifold of complex dimension bigger
 or equal to two. Then for all $k=1,\ldots , n-1$ its $k$-th Chern classes are zero: $C_k(M) =0$.
\end{lemma}

\begin{thm}\label{thm:zeroprojectivelyflat2}
 Let $(M , \, J)$ be a compact K\"ahlerian manifold. Then
 the vanishing of  the first and the second Chern classes of $M$ vanish 
 is equivalent to the existence of a zero projectively flat metric on $M$.
\end{thm}

\begin{proof}
 Assuming the vanishing of the classes, then by the Calabi Yau theorem \cite{yau}, on any K\"ahler class (which exists by assumption) of $M$ there is a K\"ahler Ricci flat metric $h$.
 Since $h$ is K\"ahler, its first and second Ricci curvature coincide, and we can write down the K\"ahler Ricci flat condition as
 $\Theta^{(2)}_{[h]} = \Lambda_h \Theta_{[h]} = 0$. This is entailing that $h$ is $h$-Hermitian Yang Mills with $\gamma=0$.
 Now, the Kobayashi L\"ubke inequality \cite[(2.2.3) page 51]{lubketeleman} tells us that, as $h$ is K\"ahler and $h$-Hermitian Yang Mills metric,
 then, denoting as usual by $\omega$ the K\"ahler form of $h$, we have
 \begin{align*}
  \int_M \left( 2n C_2 (M) - (n-1) C_1^2 (M) \right) \wedge \omega^{n-2} \geq 0
 \end{align*}
 and equality holds if and only if $h$ is projectively flat. Using the hypothesis on the Chern classes and the fact that $\omega$ is K\"ahler,
 the Stokes theorem allows to conclude that in fact $h$ is projectively flat.
 The converse statement follows from Lemma \ref{lem:preparation}.
 This completes the proof of the Theorem.
\end{proof}

\begin{rem}
 
 It would be interesting to have a statement 
 in the vein of Theorem \ref{thm:zeroprojectivelyflat2} in the non-K\"ahlerian case
 (Compare \cite[Proposition, page 67]{BourguignonindeBartolomeisTricerriVesentini}).
\end{rem}

\section{Positive projectively flat metrics}
 We begin with recalling a well known fundamental concept.
 
  \begin{defi}
   \label{defi:lcK}
   Let $(M, \, h)$ be a Hermitian manifold; let $\omega$ be the fundamental $(1,1)$-form corresponding to $h$.
   Then $h$ is called locally conformally K\"ahler if and only if for any point $p\in M$ there exists an open neighborhood $U\subset M$
   of $p$ and a function $u\in C^{\infty}(U ; \mathbb{R})$ such that $\exp(u) \cdot h$ is a K\"ahler metric on  $U$. Equivalently,
   there exists a $(1,0)$-form $\alpha$ such that there holds
   \begin{align*}
    \partial \omega = \alpha \wedge \omega \; .
   \end{align*}
  \end{defi}
  
 Thanks to the previous definition, 
 we now recall the statement of a result \cite[Lemma 2 and thereafter]{liyauzhengcag94}, 
 for which we give a proof alternative to the original one.
  
  \begin{lemma}
  \label{lemma:pfimpliesbalancedorlcK}
   Let $(M , \, h)$ be compact Hermitian manifold. If $h$ is projectively flat, then either $h$ 
   is balanced or $h$ is locally conformally K\"ahler.
  \end{lemma}
  
  \begin{proof}
   We start with arguing as in \cite[Lemma 1]{liyauzhengcag94}: let 
   $e = (e_1, \ldots , e_n)$ be a unitary frame of the holomorphic tangent bundle of $M$,
   and let $\phi = (\phi_1  ,\ldots , \phi_n)$ be its dual coframe. Let $\theta = 
   \theta ' + \theta ''$ be the connection matrix
   under $e$ decomposed into its $(1,0)$ and $(0,1)$ parts, and $\tau= (\tau_1 , \ldots , \tau_n)$ 
   be the torsion forms under $e$, 
   where $\tau_i = \frac{1}{2} \sum_{j,k=1}^n T_{jk}^i \phi_j \wedge \phi_k$, 
   and $T_{jk}^i = - T_{kj}^i$.
   Then the torsion one-form $\eta$ is defined, via 
   $\partial \omega^{n-1} = (n-1) \eta \wedge \omega^{n-1}$, by
   \[
    \eta = \frac{1}{n-1} \sum_{j,k=1}^n T_{jk}^k \phi_j \; .
   \]
   The structure equations give
   \[
    \bar\partial \phi = \phi \wedge \theta'' , \qquad 
    \bar\partial \tau = \phi \wedge \rho^{(1)} - \tau \wedge \theta'' ,
   \]
   where $\rho^{(1)} = \rho^{(1)}_{i\bar j} \phi_i \wedge \bar{\phi}_j$.
Then 
 \[
 \bar\partial \tau_i =  \frac{1}{2}\sum_{j,k, l=1}^n \left(\bar\partial_l T^{i}_{jk}\right) \phi_j \wedge \phi_k \wedge \bar\phi_l
 +  \frac{1}{2} \sum_{j,k=1}^n T^i_{jk} \bar \partial \phi_j \wedge \phi_k 
 - \frac{1}{2}  \sum_{j,k=1}^n T^i_{jk} \phi_j \wedge \partial \phi_k =
 \]
 \[
  =  \frac{1}{2} \sum_{j,k, l=1}^n \left(\bar\partial_l T^{i}_{jk}\right) \phi_j \wedge \phi_k \wedge \bar\phi_l
  +  \frac{1}{2} \sum_{a,j,k,l=1}^n T^i_{jk} \phi_a \wedge A_{aj, \bar l} \bar \phi_l \wedge \phi_k 
 -  \frac{1}{2} \sum_{b,j,k,l=1}^n T^i_{jk} \phi_j \wedge \phi_b \wedge A_{b k , \bar l} \bar \phi_l =
 \]
 \[
  =\sum_{l=1}^n \sum_{j<k} \phi_j \wedge \phi_k \wedge \bar \phi_l 
  \left(
  \bar\partial_l T^{i}_{jk}
  + \sum_{a=1}^n T^i_{aj}A_{ka, \bar l}
  - \sum_{b=1}^n T^i_{bk}A_{jb, \bar l} \, ,
  \right)
 \]
where we used the anti symmetry of $T$ in the last equality.
 On the other hand, 
 \[
  \phi_i \wedge \rho^{(1)} - \sum_{p=1}^n\tau_p \wedge {\theta ''}_{pi} = 
  \sum_{j<k} \sum_{l=1}^n \phi_j \wedge \phi_k \wedge \bar \phi_l 
  \left(
  \delta_{ij} \rho^{(1)}_{k\bar l} - \delta_{ik} \rho^{(1)}_{i \bar l}
  - T^{p}_{jk}A_{pi, \bar l} 
  \right) \, .
 \]
 From this we get, as in \cite[Lemma 1]{liyauzhengcag94},
 \[
  \bar\partial_l T^{k}_{jk} = -(n-1)\rho^{(1)}_{j\bar l}
  + T^k_{ak}A_{ja, \bar l} \, ,
 \]
which amounts to $\bar \partial \eta = \rho^{(1)}$.
 Let us consider, still following \cite{liyauzhengcag94},
the section $\sigma = (\tau - \eta \wedge \phi) \otimes e^T$ of 
$T^* M \otimes T^* M \otimes TM$.
 Again from the structure equations, using that $\bar \partial \eta = \rho^{(1)}$, 
 it follows that $\sigma$ is a holomorphic section.
  Let $H$ be the Hermitian metric on $T^* M \otimes T^* M \otimes TM$ 
  induced by the projectively flat metric $h$ on the tangent bundle.
   Let $v_1 , \ldots , v_n$ be a local holomorphic frame around $x$.
  At $x$, we request that $H_{i\bar j } = \delta_{ij}$ and $dH_{i\bar j} = 0$.
 As $\| \sigma \|^2 = \sigma_i H_{i\bar j} \sigma_{\bar j}$, since the hypothesis
 on $h$ being projectively flat implies that $H$ is projectively flat as well,
 \begin{align}\label{equa: inequality LYZ}
  \sqrt{-1} \partial \bar \partial \| \sigma \|^2_{|x} \geq \rho^{(1)} \|\sigma \|^2 \; .
 \end{align}
We would like to conclude that either $\sigma = 0$ or $\rho^{(1)}= 0$; we give now an argument alternative 
to the original one. We have three cases, according to the Gauduchon degree of the projectively flat
metric under consideration.
 The first case, of negative Gauduchon degree, does not allow any projectively flat metric 
 by means of Theorem \ref{thm:negativeprojectivelyflat}. The case of Gauduchon degree zero,
 in view of Theorem \ref{thm:zeroprojectivelyflat}, allows to conclude that $\rho^{(1)}=0$.
 Finally, in the case of positive Gauduchon degree, we apply the Gauduchon conformal method
 Lemma \ref{lem: gauduchon conformal trick}, which provides a metric $h_+$ on
 the holomorphic tangent bundle, which is still projectively
 flat being conformal to the initial projectively flat metric $h$, such that
 its scalar curvature $s_+$ is a positive function.
 Now, for  $h_+$ there holds \eqref{equa: inequality LYZ} as well, and tracing it we get to
  \[
   -\Delta_d \| \sigma \|^2 + (\eta \, , d\| \sigma \|^2)_{h_+} \geq  \| \sigma\|^2  s_+ \, ,
  \]
 where $\Delta_d$ is the Hodge Laplacian of $h_+$.
 Now, the maximum principle allows to conclude that $\sigma =0$.
  This completes the proof of the lemma.
  \end{proof}

  The next notion is going to be of primary interest in the remainder of the manuscript.
  
  \begin{defi}
   \label{defi:lcfK}
   Let $(M, \, h)$ be a Hermitian manifold.
   Then $h$ is called locally conformally flat-K\"ahler if and only if for any point $p\in M$ there exists an open neighborhood $U\subset M$
   of $p$ and a function $u\in C^{\infty}(U ; \mathbb{R})$ such that $\exp(u) \cdot h$ is both Chern flat and a K\"ahler metric on  $U$. 
  \end{defi}
  
  The following fact was already claimed in \cite[Remark (2) page 235]{vaisman}.
  
  \begin{lemma}
   \label{lemma:lcfKarelcfandlcK}
   Let $(M, \, h)$ be a compact Hermitian manifold. Then $h$ is locally conformally flat-K\"ahler if and only if it is locally conformally Chern flat
   and locally conformally K\"ahler.
  \end{lemma}
 
 \begin{proof}
  Assuming that $h$ is locally conformally flat-K\"ahler, the conclusion is obvious. 
  
  Vice versa, by means of Lemma \ref{lemma:pjequallcf} we can state the assumption as: $h$ is projectively flat and locally
  conformally K\"ahler. We pick a point $p\in M$ and since $h$ is locally conformally K\"ahler
  we get an open neighborhood $U\subset M$ of $p$ such that
  $\exp(u)\cdot h$ is K\"ahler in $U$ for some $u\in C^{\infty}(U;\, \mathbb{R})$.
  Whence, in $U$ there holds that $\exp(u) \cdot h$ is both K\"ahler and projectively flat.
  Now, by \cite[(5.4.6)]{lubketeleman} $\exp(u) \cdot h$ is Chern flat in $U$. We conclude that $\exp(u)$ is a local
  conformal transformation which makes $h$ simultaneously K\"ahler and Chern flat.
  This completes the proof of the lemma.
 \end{proof}
 
 About locally conformally flat-K\"ahler manifolds, it is crucial for us to recall the structure theorem by Vaisman \cite{vaisman}
 (see also \cite[Section 6.2]{dragomirornea} for a detailed presentation).
 
 \begin{defi}
  \label{defi:locallyconformallyflatKahlerHopf}
  Let $H$ be a finite subgroup of the unitary group $U(n)$.
  Let $\gamma_0$ a linear transformation of $\mathbb{C}^n$ which commutes with each element of $H$
  and has the form 
  \begin{align*}
   \gamma_0(z) := 
   \left(
   \rho_0 \exp \left( 2\pi \sqrt{-1}\lambda_1 \right) z^1
   , \cdots, 
   \rho_0 \exp \left( 2\pi \sqrt{-1}\lambda_n \right) z^n
   \right)
   \; ,
  \end{align*}
  where $\rho_0 \in(0,1)$, $\lambda_1  ,\cdots ,\lambda_n \in \mathbb{R}$.
  Then consider the group $G$ given by
  \begin{align} \label{equa:GforHopf}
   G:= \{ \gamma \cdot \gamma_0^k \; | \; \gamma \in H \; , \; k\in \mathbb{Z}   \} \; .
  \end{align}
 A locally conformally flat-K\"ahler Hopf manifold is the quotient 
 \begin{align*}
  (\mathbb{C}^n \backslash 0)/ G \, , 
 \end{align*}
 where $G$ is a group as described in \eqref{equa:GforHopf}.
 In particular, we only consider those $G$ such that the quotient is compact.
 \end{defi}
 
 The label locally conformally flat-K\"ahler Hopf manifold is justified by the following classical result by Vaisman \cite{vaisman}.

 \begin{prop}
   \label{prop:structuretheorem}  
  Let $(M , \, h)$ be a compact Hermitian manifold of complex dimension bigger or equal to two. 
  If $h$ is locally conformally flat-K\"ahler then
  either $(M, \, h)$ is K\"ahler and flat or 
  $M$ is a locally conformally flat-K\"ahler Hopf manifold and $h$ is globally conformally
  to the Boothby metric, whose coefficients are given by
  \begin{align*}
   h_{i \bar j} = \frac{4\delta_{ij}}{|z|^2} \; .
  \end{align*}
   
 \end{prop}

We are ready to state the main result of this section.

\begin{thm}
 \label{thm:positiveprojectivelyflat}
 Let $(M, \, h)$ be a compact Hermitian manifold. Then it is positive projectively flat
 if and only if is a locally conformally flat-K\"ahler Hopf manifold endowed with a metric 
 globally conformal to the Boothby metric.
\end{thm}

\begin{proof}
 If $(M , \, h)$ is a locally conformally flat K\"ahler Hopf manifold endowed with the Boothby metric,
then by direct computation on the Boothby metric it is manifestly a positive projectively flat metric.
 
 Vice versa, let $(M , \, h)$ be positive projectively flat. 
 By Lemma \ref{lemma:pfimpliesbalancedorlcK}, $h$ is either balanced or
 locally conformally K\"ahler. We exclude that $h$ is balanced by means of Theorem \ref{thm:zeroprojectivelyflat}; 
 whence necessarily $h$ is locally conformally K\"ahler. 
 By Lemma \ref{lemma:pjequallcf} $h$ is also locally conformally Chern flat. 
 We deduce by means of Lemma \ref{lemma:lcfKarelcfandlcK} that $h$ is locally conformally flat-K\"ahler. 
 Finally, the structure theorem Proposition \ref{prop:structuretheorem} 
 allows us to conclude that $(M , \, h)$ is a locally conformally flat-K\"ahler Hopf manifold endowed with the
 Boothby metric.
\end{proof}

\section{Applications}
In \cite{liyauzhengcag94}, where the label projectively flat metric has exactly the same meaning as in the present
manuscript, the class of similarity Hopf manifolds \cite{fried} was considered as possibly admitting projectively flat metrics. 
Our first application is a refinement of their description 
given in \cite[Theorem 1]{liyauzhengcag94}.

\begin{defi}
 \label{defi:complexsimilaritymanifolds}
A compact complex manifold $(M , \, J)$ is called similarity Hopf manifold 
if and only if it is a finite undercover of a Hopf manifold of the form
$(\mathbb{C}^n \backslash 0) / <\phi>$, where $\phi(z)= azA$ is a complex expansion: $A\in U(n), \, a>1 , \, z=(z_1 , \ldots , z_n)$.
\end{defi}

 \begin{rem}
  Let $M$ be any finite undercover of a K\"ahlerian complex torus; then, thanks to \cite{demaillyhwangpeternell} we know that $M$ is K\"ahlerian, and
  whence we can apply to it our Theorem \ref{thm:zeroprojectivelyflat2}. It would be interesting to investigate the behavior of the
  first and the second Chern classes on such manifolds. Moreover, if the flat complex torus is not K\"ahlerian, what is possible to say about its
  finite undercovers?
 \end{rem}

\begin{cor}
 \label{cor:refinement}
 The only similarity Hopf manifolds which admit projectively flat metric 
 are locally conformally flat-K\"ahler Hopf manifolds endowed with the Boothby metric.
 In particular, if the Boothby metric is not invariant with respect to some $G$
 as in Definition \eqref{defi:lcfK}, then such locally conformally flat-K\"ahler 
 Hopf manifold is not projectively flat.
\end{cor}

\begin{proof}
 Let $M$ be a similarity Hopf manifold with projectively flat metric $h$.
 Assume by contradiction that $h$ is zero projectively flat.
 Then, by Theorem \ref{thm:zeroprojectivelyflat}, we have that $h$ is balanced. 
 Now, pulling back $\omega$ (the fundamental $(1,1)$-form of $h$) to the base Hopf manifold
 via the finite covering map, 
 we get a balanced metric on a Hopf manifold which is excluded in \cite{michelsohn}.
 So we have that $h$ is positively projectively flat, and we apply \eqref{thm:positiveprojectivelyflat} 
 to get the claimed structure of $(M,g)$.
 The last sentence of the statement is obvious.
\end{proof}

\begin{rem}
 We wouldn't be surprised to learn that the previous result was already stated in literature. We acknowledge that
  in \cite[pages 180-181]{lubketeleman} the authors refer to \cite{jostyau,liyauzhengcag94} 
  for a description of projectively flat metrics in dimension bigger or equal to three.
\end{rem}

The next application deals with projectively 
flat astheno-K\"ahler metrics (compare \cite[Corollary 3]{liyauzhengcag94}).

\begin{defi}
 \label{defi:asthenokahler}
 Let $(M, \, h)$ be a Hermitian manifold; let $\omega$ be the fundamental $(1,1)$-form corresponding to $h$.
 Then $h$ is called astheno-K\"ahler metric if and only if $\partial\bar \partial (\omega^{n-2}) = 0$.
\end{defi}

\begin{cor}
\label{cor:asthenoarekahler}
 Let $(M, \, h)$ be a compact Hermitian manifold of complex dimension bigger or equal to three. 
 If $h$ is astheno-K\"ahler and projectively flat, then $h$ is K\"ahler and zero projectively flat.
\end{cor}

\begin{proof}
 In view of Theorems \ref{thm:negativeprojectivelyflat}, \ref{thm:zeroprojectivelyflat}, and \ref{thm:positiveprojectivelyflat}
 we have to exclude that $M$ is a locally conformally K\"ahler flat Hopf manifold and that $M$ is non-K\"ahlerian zero projectively flat.
 About the first claim, by \cite{vaisman} we have an explicit expression of the form of such metrics on those Hopf manifolds, 
 which are not astheno-K\"ahler. 
 Whence, now we know that $h$ is zero projectively flat. 
 About the second claim, by Theorem \ref{thm:zeroprojectivelyflat},
 we know that $h$ is balanced. Then, by \cite{matsuotakahashi} we conclude that $h$ is K\"ahler.
 This completes the proof of the corollary.
\end{proof}

\begin{rem}
 Since any Hermitian metric on a Riemann surface or a complex surface is astheno-K\"ahler, Corollary \ref{cor:asthenoarekahler} doesn't
 hold in dimension one and two. In particular, in dimension two K\"ahler projectively flat metrics are zero projectively flat; whence all
 the counterexamples in dimension two are given by locally conformally flat-K\"ahler Hopf manifolds.
\end{rem}

\begin{rem}
 Since any Hermitian metric on a Riemann surface or a complex surface is astheno-K\"ahler, Corollary \ref{cor:asthenoarekahler} doesn't
 hold in dimension one and two. In particular, in complex dimension two K\"ahler projectively flat metrics are zero projectively flat; whence all
 the counterexamples in dimension two are given by locally conformally flat-K\"ahler Hopf manifolds.
\end{rem}

\end{document}